\documentclass{amsart}
\usepackage{graphicx, amsmath, amssymb,amsthm, geometry, caption}
\usepackage{indentfirst} 

\newtheorem{Thm}{Theorem}
\newtheorem{lem}{Lemma}
\newtheorem{prop}{Proposition}
\newtheorem{remark}{Remark}

\newtheorem{defn}{Definition}

\def\*#1{\mathbf{#1}} 

\begin{document}
\title{Cauchy surface area formula in the Heisenberg groups}
\author{Yen-Chang Huang}
\address{National University of Tainan, Tainan city, Taiwan R.O.C.}
\curraddr{}
\email{ychuang@mail.nutn.edu.tw}


\dedicatory{}

\begin{abstract}
We show an analogy of Cauchy's surface area formula for the Heisenberg groups $\mathbb{H}_n$ for $n\geq 1$, which states that the p-area of any compact hypersurface $\Sigma$ in $\mathbb{H}_n$ with its p-normal vector defined almost everywhere on $\Sigma$ is the average of its projected p-areas onto the orthogonal complements of all p-normal vectors of the Pansu spheres (up to a constant). The formula provides a geometric interpretation of the p-areas defined by Cheng-Hwang-Malchiodi-Yang \cite{CHMY1} in $\mathbb{H}_1$ and Cheng-Hwang-Yang \cite{CHY} in $\mathbb{H}_n$ for $n\geq 2$. We also characterize the projected areas for rotationally symmetric domains in $\mathbb{H}_n$, namely, for any rotationally symmetric domain with boundary in $\mathbb{H}_n$, its projected p-area onto the orthogonal complement of any normal vector of the Pansu spheres is a constant, independent of the choices of the projected directions.
\end{abstract}
\maketitle

\section{Introduction}
%

One of the most important problems in Integral Geometry is how to obtain geometric information (for instance, the lengths, surface areas, and volumes) of objects by the lower dimensional geometric quantities and Cauchy's surface area formula in $\mathbb{R}^n$ provides one of the solutions to this problem. It says that the surface area of any convex body in the $n$-dimensional Euclidean space is equal to the average of the areas of its orthogonal projections onto all subspaces. By a convex body we mean a compact convex set in $\mathbb{R}^n$ with non-empty interior. The formula was originally proved by Augustin Cauchy in 1841 \cite{Cau41} and in 1850 \cite{Cau50} for $n = 2, 3$, respectively. After that, the formula was generalized by Kubota \cite{Kub25}, Minkowski \cite{Min}, and Bonnesen \cite{Bon29}. The literature on the formula and its applications to other scientific fields are quite large, for instance, we refer the interested readers to \cite{KR,Gardner,Schneider,Santalo} for the viewpoints of Integral Geometry and Convex Geometry, and \cite{Vouk}, \cite[p 290]{MTL} for the applications to the measurement of elementary particles in chemistry. Recently a simple proof of the formula was presented in \cite{TV} by a key observation of algebraic property for Minkowski sum of sets in $\mathbb{R}^n$.

Let $K$ be any compact convex subset in $\mathbb{R}^n$. If $V$ is an $(n-1)$-dimensional subspace of $\mathbb{R}^n$, denote $K|V$ by the orthogonal projection of $K$ onto $V$ and $\omega_{n-1}$ by the $(n-1)$-dimensional volume of the unit ball in $\mathbb{R}^{n-1}$. At first, let us recall the classical result derived by Cauchy:

\begin{Thm}[Cauchy's surface area formula, Theorem 5.5.2 \cite{KR}]\label{cauchyrn}
For any $n$-dimensional compact convex subset $K$ in $\mathbb{R}^n$, denote the surface area $S(K)$ of $K$, then we have
\begin{align}
S(K)=\frac{1}{\omega_{n-1}}\int_{u\in\mathbb{S}^{n-1}}\mu_{n-1}(K| u^\perp) dS_u.
\end{align}
Here $\mathbb{S}^{n-1}$ denotes the $(n-1)$-dimensional standard unit sphere, $dS_u$ is the surface area element at $u\in \mathbb{S}^{n-1}$, and $\mu_{n-1}(K|u^\perp)$ is the $(n-1)$-dimensional volume of the orthogonal projection of $K$ onto the subspace $u^\perp=\{v\in \mathbb{R}^n; v \text{ is perpendicular to } u \}$.
\end{Thm}

The term $\mu_{n-1}(K| u^\perp)$ is the projected area of $K$ onto the subspace in $\mathbb{R}^n$ perpendicular to the vector $u$ and it can be explicitly represented by the integral formula
\begin{align}\label{projecteda}
\mu_{n-1}(K| u^\perp)=\int_{v\in \mathbb{S}^{n-1}} |u\star v|dS_v,
\end{align}
where $v$ is the outward unit normal vector to the surface of $K$, $dS_v$ the surface area element at $v\in \mathbb{S}^{n-1}$, and $u\star v$ is the standard inner product of vectors $u$ and $v$ in $\mathbb{R}^n$. The identity \eqref{projecteda} can be proved by applying the discrete form of the similar identity on an $n$-polygon and taking the limit as the value $n$ goes to infinity. See \cite[p 56]{KR} for the detail. In particular, when $K=\mathbb{S}^{n-1}$, the following lemma shows that the surface area of the orthogonal projection of the standard unit sphere is independent of the choices of the projected direction. This natural property for $\mathbb{S}^{n}$ plays the key role for the proof of Theorem \ref{cauchyrn}. Basically the main result in the present paper (see Theorem \ref{mainthm1}) follows a similar concept from Lemma \ref{KR}.

\begin{lem}[\cite{KR} Lemma 5.5.1]\label{KR}
For any $u\in \mathbb{S}^{n-1}\subset \mathbb{R}^n$,
\begin{align}\label{KR1}
\mu_{n-1}(\mathbb{S}^{n-1}|u^\perp)=\int_{v\in \mathbb{S}^{n-1}} |u\star v| dS_v = 2\ \omega_{n-1},
\end{align}
where $u\star v$ is the standard inner product of vectors $u,v$ in $\mathbb{R}^n$ and $dS_v$ is the surface area element at $v\in \mathbb{S}^{n-1}$.
\end{lem}
Notice that since the projection of the sphere onto the $(n-1)$-dimensional subspace $u^\perp$ is counted twice (from the "front" and the "back" of the subspace $u^\perp$), the coefficient $2$ comes in as shown on the right-hand side of the identity \eqref{KR1}. Roughly speaking, the identities \eqref{projecteda} and \eqref{KR1} can be realized in the following geometric intuition: since the integrand $|u\star v|$ on the right-hand side of those identities are the absolve value of the cosine angle between the vectors $u$ and $v$, the quantity $|u\star v|dS_v$ is the projected area of the infinitesimal area at $v\in \mathbb{S}_{n-1}$ onto the plane $u^\perp$.

We also mention that the reversed statement of Lemma \ref{KR} is not true in general. To be precise, a convex body in $\mathbb{R}^n$ with equal projected areas along any direction is \textit{not} necessarily a sphere. For instance, a Reuleaux triangle in $\mathbb{R}^2$ is not a circle but a convex body with constant widths and equal projected areas along any directions.

Here is the key observation in the proofs of Theorem \ref{cauchyrn} and Lemma \ref{KR}. Both proofs are based on the invariant translations in the homogeneous space $\mathbb{R}^n$. Suppose $u(0)$ is the direction that $K$ (or $\mathbb{S}^{n-1}$ for Lemma \ref{KR}) is projected along, by the inner product $u\star v$ in \eqref{KR1} we mean that the value $u(p)\star v(p)$ is taken at the boundary point $p\in\partial K$ of $K$ in the tangent space $T\mathbb{R}^n_p$ of $\mathbb{R}^n$ at $p$ by considering the parallel transport of the vector $u(0)$ at the origin to the outward unit normal $u(p)$ on the boundary $\partial K$. Hence when considering the Euclidean space $\mathbb{R}^n$ as a Lie group with the natural left translation $L_p q:=p+q$ for any points $p$ and $q$ in $\mathbb{R}^n$, we can write
\begin{align}\label{keyob2}
|u(p)\star v(p) |=|L_{p_*}u(0)\star v(p)|,
\end{align}
where $L_{p_*}$ is the pushforward of the left translation $L_p$. In the present work, we will show the analogies of Theorem \ref{cauchyrn} and Lemma \ref{KR} respectively in the approach of the Heisenberg groups, which are regarded as the flat models of pseudohermitian manifolds. The notion for parallel transports in $\mathbb{R}^n$ shall be replaced by the left invariant translations in the Heisenberg groups $\mathbb{H}_n$ for $n\geq 1$. Unlike using the standard unit sphere $\mathbb{S}^{n-1}$ in $\mathbb{R}^n$ to derive the Cauchy's surface area formula (namely, Lemma \ref{KR}), the Pansu spheres (defined in next section, equation \eqref{Pansu}) see to be a more natural model in $\mathbb{H}_n$ for the formula. See Remark \ref{geointerprtation} for a geometric interpretation of this approach. We also mention that the approach of group translations weaken the assumption for the convexity of domains in the Euclidean spaces. Here we only need that the boundary is of class $C^2$ such that the most of p-normal vectors can be defined on the boundary. See the paragraph before Definition \ref{projectdef} below.

Now we give a brief introduction about some background (see next section for more details). The Heisenberg group $\mathbb{H}_n$ for $n\geq 1$ is defined as the Euclidean space $\mathbb{R}^{2n+1}$ with contact structure $\xi$ (also called \textit{distribution}). At each point in $\mathbb{H}_n$ there is a contact plane $\xi$ of dimensional $2n$. Suppose $\Sigma$ is a hypersurface in $\mathbb{H}_n$. The singular set of $\Sigma$ is a set of points of $\Sigma$ at which the tangent plane of $\Sigma$ coincides with the contact plane; otherwise, points on $\Sigma$ are called the nonsingular points. At any nonsingular point, the intersection of the tangent plane $T\Sigma$ and the contact plane $\xi$ is of dimension $2n-1$ and so there is a unique unit vector $N$ (called the \textit{p-normal vector}) contained in $\xi$ and perpendicular to the intersection $T\Sigma\cap \xi$ with respect to the Levi metric. The volume elements and the area elements in $\mathbb{H}_n$ are the usual Euclidean volumes and the \textit{p-areas}, respectively. The p-area, introduced by Cheng-Hwang-Yang \cite{CHY} for $n=1$ and Cheng-Hwang-Malchiodi-Yang \cite{CHMY1} for $n\geq 2$, comes from a variation of the surface $\Sigma$ in the normal direction $f N$ for some suitable function $f$ with compact support on the regular points of $\Sigma$. Such the invariant area measure in $\mathbb{H}_n$ coincides with the three dimensional Hausdorff measure of $\Sigma$, considered in \cite{FS}, \cite{Balogh}, and \cite{Balogh2}. Notice that although the p-normal vectors are not defined at the singular points of $\Sigma$, the p-area element is globally defined on $\Sigma$ and vanishes at the singular points (\cite[p 135]{CHMY1}). Morever, the p-area element can be represented explicitly in a variety of forms, including the differential forms (see Appendix in \cite{CHMY1} for $\mathbb{H}_n$), the local coordinates for graphs \cite[p 261]{CHY}, and recently for parametrized surfaces \cite{H2} by the author. These notions, especially in the framework of geometric measure theory, have been used to study existence or regularity properties of minimizers for the relative perimeter
or extremizers of isoperimetric inequalities (see, e.g., \cite{RR, RR1, DGN, GN, LM, LR, MR, HR, Monti}).

Recall that \cite{Pansu} the Pansu spheres $\mathcal{P}_\lambda^n$ in $\mathbb{H}_n$ can be defined by rotating a (helix) geodesic joining the points $(0,\cdots, 0, \pm \frac{\pi}{4\lambda^2})\in\mathbb{R}^{2n+1}$ (the "north" and "south" poles respectively) about the $x_{2n+1}$-axis such that its "equator" is a standard sphere $\mathbb{S}^{2n-1}_{1/\lambda}$ in $\mathbb{R}^{2n}$ centered at the origin with radius $\frac{1}{\lambda}$. It is rotationally symmetric and topologically equivalent to the standard unit sphere $\mathbb{S}^{2n}$ and its parametrization can be exactly expressed as in \eqref{Pansu}. Monti in \cite{Monti2} analyzed the symmetrization of $\mathcal{P}^n_\lambda$, and Cheng et al. \cite{CCHY} used a notion of umbilicity to characterize the Pansu spheres in $\mathbb{H}_n$ for $n\geq 2$. There are only two singular points occurred at the north and the south poles of $\mathcal{P}_\lambda^n$, and so the p-normal vectors are well-defined globally on $\mathcal{P}_\lambda^n$ except for the poles. In Theorem \ref{mainthm1} and Proposition \ref{keyprop} we will show that the Pansu spheres in $\mathbb{H}_n$ play a natural analogy as the standard unit spheres to the Euclidean spaces as shown in Theorem \ref{cauchyrn} and Lemma \ref{KR}.

We make a remark for the requirements of the hypersurfaces we concerned. By Theorem B in \cite{CHMY1}, the authors assume that the surfaces $\Sigma$ in $\mathbb{H}_1$ are of class $C^2$-regularity with bounded p-mean curvatures and have the result that the set of all singular points in $\Sigma$ consists of only isolated points and smooth curves. In both cases, the measure of the singular set is zero in the sense of p-area elements, and so this set does not influence the results of the integrals for the p-normal vectors even though the p-normals are not defined on the singular sets (for example, the right-hand sides of \eqref{def1} and \eqref{anydirection2}). Therefore, to our purpose, whenever in the article we say "\textit{the normal vectors are defined on the hypersurface $\Sigma$ a.e.}", it means that the normal vectors are globally defined on $\Sigma$ except possibly for the set of measure zero in the sense of p-area elements. Besides, the inner product of vectors in $\mathbb{H}_n$ is always adopted with respect to the Levi metric and denote $\tilde{N}(p)$ (resp. $N(p)$) by the unit p-normal vector at $p$ on the Pansue sphere $\mathcal{P}^n_\lambda$ (resp. any hypersurface $\Sigma$) in $\mathbb{H}_n$.

The following definition was motivated by Lemma \ref{KR} and the identity \eqref{keyob2} for compact hypersurfaces in the Heisenberg groups $\mathbb{H}_n$.
\begin{defn}\label{projectdef}
Let $\Sigma$ be any compact hypersurface in the Heisenberg group $\mathbb{H}_n$, $n\geq 1$, with the p-normal vectors defined a.e on $\Sigma$. Given a unit p-normal vector $\tilde{N}(p)$ at $p$ on the Pansu sphere $\mathcal{P}_\lambda^n$. The projected p-area of $\Sigma$ onto the orthogonal complement $\tilde{N}(p)^\perp:=\{u\in T_p\mathbb{H}_n;  u\cdot \tilde{N}(p)=0 \}$ is defined by
\begin{align}\label{def1}
\mathcal{A}(\Sigma|\tilde{N}^\perp (p))=\int_{q\in \Sigma} |L_{qp^{-1}*}\tilde{N}(p)\cdot N(q)| d\Sigma_q,
\end{align}
where $L_{qp^{-1}*}$ is the pushforward of the left translation $L_{qp^{-1}}:=L_q \circ L_{p^{-1}}$ in $\mathbb{H}_n$ and $N(q)$ the unit p-normal vector of $\Sigma$ at $q\in \Sigma$.
\end{defn}

In Definition \ref{projectdef} we only consider the projected p-areas of the surfaces $\Sigma$ onto the p-normal vectors of the Pansu spheres $\mathcal{P}_\lambda^n$. There are two reasons: first, by applying the pushforward $L_{{p^{-1}}_*}$ of the left translation $L_{p^{-1}}$ on the unit p-normal vector $\tilde{N}(p)$ of $\mathcal{P}_\lambda^n$, it is clear that the set $\{L_{{p^{-1}}_*}\tilde{N}(p)|\text{ all } p\in \mathcal{P}_\lambda^n\}= \mathbb{S}^{2n-1}\subset \mathbb{R}^{2n}\subset \mathbb{H}_n$, and hence the directions that the projection is along with are not full of all possible positions in $\mathbb{H}_n$ (namely, $\mathbb{S}^{2n}$). This raises a question of considering the projected p-areas along \textit{any} directions, namely, along any vector $u\in \mathbb{S}^{2n}$. However, the following Proposition \ref{anydirection} shows that for $n=1$ when $\Sigma$ is a rotationally symmetric surface and we consider the projections of $\Sigma$ along arbitrary directions $u$ in $\mathbb{H}_1$, the projected p-areas of $\Sigma$ depend on the choices of the projected directions. As a result, the projected p-areas of both the Pansu spheres and the standard unit sphere $\mathbb{S}^{2}$ will \textit{not} be the constants if the projected directions are in $\mathbb{S}^{2n}$. The observation suggests that using the p-normal vectors of the Pansu spheres as the projected vectors would be better than that of the usual Euclidean normal vectors. Secondly, the projected p-areas of the Pansu spheres along its p-normal vectors is a constant, independent of the choices of the projected directions and the value can be exactly solved (Proposition \ref{keyprop}), but that of the spheres $\mathbb{S}^{2n}$ involve an integral which does not seem to have the closed-form \eqref{integralbounded}. Therefore, the definition of considering the p-normal vectors of the Pansu spheres is a natural generalization of projected p-areas in the Heisenberg groups.

The following result shows that the projected p-areas of any rotationally symmetric compact surface depend on the projected directions if we consider \textit{arbitrary} directions in $\mathbb{S}^{2}\subset \mathbb{H}_1$.

\begin{prop}\label{anydirection}
Let $\Sigma$ be any  rotationally symmetric compact hypersurface $\Sigma$ in the Heisenberg group $\mathbb{H}_1$ with the p-normal vectors defined a.e on $\Sigma$. Suppose $\Sigma=\Sigma^+\cup \Sigma^-$ can be represented by
$$\Sigma^\pm:(r,\theta)\rightarrow (r\cos\theta, r\sin\theta,h^\pm(r)),$$
for some functions $h^+\geq 0$, $h^-\leq 0$, and
$0\leq \theta \leq 2\pi$, $0\leq r\leq R$ for some positive number $R$.
If $h^+$ and $h^-$ satisfy any of the following conditions:
\begin{enumerate}
\item  $\frac{d h^+}{d r}$ or $\frac{d h^-}{d r}$ is continuous on the interval $[0,R]$,
\item  $|\frac{d h^+}{d r}|$ or $|\frac{d h^-}{d r}|$ is bounded by $\frac{r}{\sqrt{R^2-r^2}}$ on the interval $[0,R)$.
\end{enumerate}
Then for any vector $u(0):=(\sin\alpha\cos\beta, \sin\alpha\sin\beta, \cos\alpha)$ in $\mathbb{S}^2$, $0\leq \alpha \leq \pi, 0\leq \beta \leq 2\pi$, the projected p-area of $\Sigma$ onto the subspace $u(0)^\perp$ perpendicular to $u(0)$ is given by
\begin{align}\label{anydirection2}
\mathcal{A}(\Sigma|u(0)^\perp)=\int_{q\in \Sigma} |L_{q_*}u(0) \cdot N(q)| d\Sigma_q = |\sin\alpha|C,
\end{align}where $N(q)$ is the unit p-normal vector at $q\in \Sigma$ and $C$ is a constant independent of the choices of the vector $u(0)$. In particular, when $\alpha=\frac{\pi}{2}$, the identity \eqref{anydirection2} becomes to $\eqref{rotapconstant}$ for $n=1$.
\end{prop}

Notice that in Proposition \ref{anydirection} the sharp case occurs, namely $h_r=\frac{r}{\sqrt{R^2-r^2}}$, if and only if the compact hypersurface $\Sigma$ is the standard sphere $\mathbb{S}^2(R)$ with radius $R$ in $\mathbb{H}_1$. It is similar for the case in Theorem \ref{rotationalpop} when $\Sigma=\mathbb{S}^{2n}(R)$, the standard sphere with radius $R$ in $\mathbb{H}_{n}$.

Our main result is an analogy of Theorem \ref{cauchyrn} in the Heisenberg groups, which shows that the p-area $\mathcal{A}(\Sigma)$ of any compact convex surface $\Sigma\subset \mathbb{H}_n$ is the average of the projected p-areas onto the orthogonal complements of all unit p-normal vectors in $\mathcal{P}^n_\lambda$ over the volume of the $(2n-1)$-dimensional Euclidean sphere.

\begin{Thm}\label{mainthm1}
Given any compact hypersurface $\Sigma$ in the Heisenberg group $\mathbb{H}_n$, $n\geq 1$, with the p-normal vectors defined a.e on $\Sigma$. Let $\tilde{N}(p)$ be a unit p-normal vector at $p$ on the Pansu sphere $\mathcal{P}_\lambda^n$ and $\mathcal{A}(\Sigma | \tilde{N}(p)^\perp)$ the projected p-area of $\Sigma$ onto the orthogonal complement $\tilde{N}(p)^\perp$. Then the p-area of $\Sigma$ is given by
\begin{align}\label{formula1}
\mathcal{A}(\Sigma)=\frac{1}{2 C_n \omega_{2n-1}} \int_{p\in \mathcal{P}_\lambda^n} \mathcal{A}(\Sigma | \tilde{N}(p)^\perp) d\Sigma_p,
\end{align}
where $C_n=\frac{\sqrt{\pi}}{\lambda^{2n+1}}\frac{\Gamma(n+\frac{1}{2})}{\Gamma(n+1)}$ is a dimensional constant, $\Gamma(x)$ the Gamma function, and $\omega_{2n-1}$ the volume of the $(2n-1)$-dimensional Euclidean sphere. Moreover, the p-area of the surface $\Sigma$ can be represented by the average projected p-areas of $\Sigma$ along all p-normal vectors of the Pansu sphere, namely,
\begin{align}\label{formula2}
\mathcal{A}(\Sigma)=\frac{S_{2n-1}}{2 \omega_{2n-1}\mathcal{A}(\mathcal{P}^n_\lambda)}\int_{p\in \mathcal{P}^n_\lambda} \mathcal{A}(\Sigma| \tilde{N}(p)^\perp) d\Sigma_p,
\end{align}where $S_{2n-1}$ is the (Euclidean) surface area of the $(2n-1)$-sphere in $\mathbb{R}^{2n}$.
\end{Thm}

The proof of Theorem \ref{mainthm1} is based on the following observation that the projected p-area of the Pansu sphere $\mathcal{P}_\lambda^n$ in $\mathbb{H}_n$ along its p-normal vectors is a constant; it is an analogous property of Lemma \ref{KR} for spheres in $\mathbb{R}^n$.

\begin{prop}\label{keyprop}
Given the Pansu sphere $\mathcal{P}^n_\lambda$ in $\mathbb{H}_n$ for $n\geq 1$ defined by \eqref{Pansu} and any p-normal vector $\tilde{N}(p)$ for some fixed $p\in \mathcal{P}^n_\lambda$. Denote $\tilde{N}(q)$ and $L_{qp^{-1}*}\tilde{N(}p)$ respectively by the unit p-normal vector at $q\in \mathcal{P}^n_\lambda$ and the pushforward of the left translation $L_{qp^{-1}}$ operated on $\tilde{N}(p)$. Then the projected p-area of the Pansu sphere along $\tilde{N}(p)$ is a dimensional constant, namely,
\begin{align}\label{pconstant}
\mathcal{A}(\mathcal{P}_\lambda^n|\tilde{N}(p)^\perp)=\int_{q\in \mathcal{P}^n_\lambda} |L_{qp^{-1}*}\tilde{N}(p) \cdot \tilde{N}(q)| d\Sigma_q = 2 C_n\omega_{2n-1},
\end{align}where $C_n$ is the constant defined as the one in Theorem \ref{mainthm1}.
\end{prop}

In fact, Proposition \ref{keyprop} can be generalized to any rotationally symmetric surfaces about the $x_{2n+1}$-axis in $\mathbb{H}_n$. This also shows that the reversed statement of Proposition \ref{keyprop} is \textit{not} true in general, which is the same result as the counterexamples of constant widths (e.g. Reuleaux triangles mentioned in the previous paragraphs) in the Euclidean spaces.

\begin{Thm}\label{rotationalpop}
Let $\Sigma=\Sigma^+\cup \Sigma^-$ be a rotationally symmetric compact hypersurface in $\mathbb{H}_n$ obtained by rotating a hypersurface in $\mathbb{R}^{2n}$ about the $x_{2n+1}$-axis, where $\Sigma^+=\{x_{2n+1}=h^+\geq 0\}$ and $\Sigma^-=\{x_{2n+1}=h^-\geq 0\}$ for some functions $h^\pm=h^\pm(r)$ defined on $[0,R]$ for some $R>0$. If $h^+$ and $h^-$ satisfy any of the following conditions
\begin{enumerate}
\item  $\frac{d h^+}{d r}$ or $\frac{d h^-}{d r}$ is continuous on the interval $[0,R]$,
\item  $|\frac{d h^+}{d r}|$ or $|\frac{d h^-}{d r}|$ is bounded by $\frac{r}{\sqrt{R^2-r^2}}$ on the interval $[0,R)$.
\end{enumerate}
and denote $\tilde{N}(p)$ by any unit p-normal vector at $p$ in the Pansu sphere $\mathcal{P}_\lambda^n$ defined by \eqref{Pansu}, then the projected p-area of $\Sigma$ onto $\tilde{N}(p)^\perp$ is given by
\begin{align}\label{rotapconstant}
\mathcal{A}(\Sigma|\tilde{N}(p)^\perp)=\int_{q\in \Sigma} |L_{qp^{-1}*}\tilde{N}(p) \cdot N(q)| d\Sigma_q = C,
\end{align}where $N(q)$ is the unit p-normal vector at $q\in \Sigma$ and $C$ is a constant independent of the choices of $\tilde{N}(p)$. In particular, if $\Sigma=\mathcal{P}^n_\lambda$, then the projected p-area is exactly same as the one in Proposition \ref{keyprop}.
\end{Thm}

Since the p-areas for Pansu sphere $\mathcal{P}_\lambda^n$ can be obtained by the equation \eqref{pansuarea} below, an immediate application of Theorem \ref{mainthm1} is that the expected value of the function $F_{\tilde{N}(p)}(\Sigma):= \mathcal{A}(\Sigma| \tilde{N}(p)^\perp)$ can be obtained by
\begin{align*}
Exp(F_{\tilde{N}(p)}):= \frac{\int_{p\in \mathcal{P}_\lambda^n} \mathcal{A}(\Sigma| \tilde{N}(p)^\perp) d\Sigma_p}{\mathcal{A}(\mathcal{P}_\lambda^n)}=\mathcal{A}(\Sigma)\omega_{2n-1}S_{2n-1}.
\end{align*}
Roughly speaking, the number $Exp(F_{\tilde{N}(p)}(\Sigma))$ is the average projected p-areas of the surface $\Sigma$ onto a randomly-chosen plane $\tilde{N}(p)^\perp$.

This paper is organized as follows. Section \ref{preliminary} recall some fundamental background about the Heisenberg groups regarded as pseudohermitian manifolds. The precise expressions for the Pansu spheres, the p-normal vectors for rotationally symmetric surfaces, and the p-areas will also be derived. Section \ref{allproofs} will show the proofs of our results and explains the geometric meanings of Proposition \ref{keyprop} in Remark \ref{geointerprtation}.

\textbf{Acknowledgement} The author thanks Professor Hung-Lin Chiu for a useful suggestion about considering the rotationally symmetric surfaces. This work was supported by Ministry of Science and Technology, Taiwan, with the grant number: 108-2115-M-024-007-MY2.

\section{Preliminary}\label{preliminary}
Let $(M, J,\Theta)$ be a $(2n + 1)$-dimensional pseudohermitian manifold with an integrable CR structure $J$ and a global contact form $\Theta$ such that the bilinear form
$G := \frac{1}{2}d\Theta(\cdot, J\cdot)$ is positive definite on the contact bundle $\xi:= ker\Theta$. The metric $G$ is usually called the Levi metric. Consider a hypersurface $\Sigma\subset M$. A point $p \in \Sigma$ is called singular if $\xi$ coincides with $T\Sigma$ at $p$. Otherwise, $p$ is called nonsingular and $V :=\xi \cap T\Sigma$ is $(2n-1)$-dimensional in this case. There is a unique (up to a sign) unit vector $N \in \xi$ that is perpendicular to $V$ with respect to the Levi metric $G$. We call $N$ the \textit{Legendrian normal} or the \textit{p-normal vector} (”p” stands for ”pseudohermitian”). Suppose that $\Sigma$ bounds a domain $D$ in $M$. The authors \cite{CHY} defined the p-area $2n-$form $d\Sigma$ by computing the first variation, away from the singular set, of the standard volume $\Theta\wedge (d\Theta)^n$ along the p-normal vector $N$:

\begin{align}
\delta_{fN}\int_D \Theta\wedge (d\Theta)^n=c(n)\int_\Sigma fd\Sigma,
\end{align}
where $f$ is a $C^\infty$-smooth function on $\Sigma$ with compact support away from the singular points, and $c(n)=2^n n!$ is a normalization constant. The sign of $N$ is determined by requiring that $d\Sigma$ is positive with respect to the induced orientation on $\Sigma$. Notice that the p-area $2n$-form can continuously extend to the set of singular points and vanish on the set, so that we can talk about the p-area of $\Sigma$ by integrating $d\Sigma$ over $\Sigma$ \cite[p 135]{CHMY1}.

One of the most interesting examples for pseudohermitian manifolds is the $n$-dimensional Heisenberg group $\mathbb{H}_n$, which can be regarded as a flat pseudohermitian manifold $(\mathbb{R}^{2n+1}, \Theta_0, J_0)$. Two surrey articles \cite{Cheng, Yang} gave some recent development on geometric analysis in $\mathbb{H}_n$. Here $\Theta_0:=  dz+ \sum_{j=1}^{n} (x_jdy_j-y_jdx_j)$ at a point $(\vec{X},z):=(x,y,z):=(x_1, y_1,\cdots, x_n, y_n, z)\in \mathbb{R}^{2n+1}$ and $J_0(\mathring{e}_{x_j})= \mathring{e}_{y_j}$, $J_0(\mathring{e}_{y_j})=-\mathring{e}_{x_j}$ where
\begin{align*}
\mathring{e}_{x_j}=\frac{\partial}{\partial x_j}+y_j \frac{\partial}{\partial z},\
\mathring{e}_{y_j}=\frac{\partial}{\partial y_j}-x_j \frac{\partial}{\partial z}
\end{align*}
for $j = 1, \cdots, n$, span the contact plane $\xi_0 := ker \Theta_0$. Notice that $\mathring{e}_{x_j}$'s and $\mathring{e}_{y_j}$'s form an orthonormal basis with respect to the Levi metric $G_0:= (\sum_{j=1}^{n} dx_j \wedge dy_j)(\cdot, J_0 \cdot)$. $\mathbb{H}_n$ is also a Lie group with the natural left translation defined by
\begin{align*}
L_{(\vec{X},z)}(\vec{X'},z'):&=L_{(x_1,y_1,\cdots, x_n,y_n, z)} (x'_1, y'_1, \cdots, x'_n, y'_n, z')\\
&=\Big(x_1+x'_1,y_1+y'_1,\cdots,x_n+x'_n, y_n+y'_n, z+z'+\sum_{i=1}^{n}(y_ix'_i-x_iy'_i)\Big).
\end{align*}

Recall that the p-area form $d\Sigma_p$ at the point $p=(x_1,y_1,\cdots,x_n,y_n,z) \in \mathbb{H}_n$ for the graph $z=f(x,y)$ (see equation (2.7) \cite{CHY}) is given by
\begin{align}\label{parea}
d\Sigma_p=D(p) dx_1 dy_1\cdots dx_n dy_n,
\end{align}
where
\begin{align}\label{dform}
D(p)=\left[\sum_{j=1}^{n} (f_{x_j}-y_j)^2+ (f_{y_j}+x_j)^2\right] ^{\frac{1}{2}}.
\end{align}
In general (see \cite[p 260]{CHY}) the p-normal vector $N(p)$ for any graph $z=f(x,y)$ at the point $p=(x_1,y_1,\cdots,x_n,y_n,z)$ is defined by
\begin{align}\label{pnormal1}
N(p)=\frac{-1}{D(p)}\sum_{j=1}^{n}\left[ (f_{x_j}-y_j)\mathring{e}_{x_j}(p)+ (f_{y_j}+x_j)\mathring{e}_{y_j}(p)\right].
\end{align}

We also recall that the Pansu sphere $\mathcal{P}^n_\lambda$ \cite{Pansu} with radius $\frac{1}{\lambda}$ in $\mathbb{H}_n$ is the union of the graphs of the functions $f$ and $-f$,
where
\begin{align}\label{Pansu}
f(x,y)=\frac{1}{2\lambda^2}\left( \lambda \sqrt{x^2+y^2} \sqrt{1-\lambda^2(x^2+y^2)}+\cos^{-1}\lambda \sqrt{x^2+y^2} \right),
\end{align}
where $\sqrt{x^2+y^2}\leq \frac{1}{\lambda}$, $x:=(x_1,\cdots, x_n), y:=(y_1,\cdots, y_n)$, $x^2:=\sum_{j=1}^{n}x_j^2$, $y^2:=\sum_{j=1}^{n}y_j^2$. Note that the intersection of the Pansu sphere $\mathcal{P}^n_\lambda$ and the plane $\{z=0\}$ is the standard sphere $\mathbb{S}^{2n-1}_{1/\lambda}$ centered at the origin with radius $\frac{1}{\lambda}$.
Denote by $r=\sqrt{x^2+y^2}=\sqrt{\sum_{j=1}^{n}(x_j^2+y_j^2)}$. Clearly, take the partial derivatives with respect to $x_j$ and $y_j$ respectively for $j=1, \cdots, n$, we have
\begin{align}\label{partialxy}
f_{x_j}:&=\frac{\partial f}{\partial {x_j}}(x,y)=\frac{-\lambda x_j r}{\sqrt{1-\lambda^2 r^2}},  \\ f_{y_j}:&=\frac{\partial f }{\partial {y_j}}(x,y)=\frac{-\lambda y_j r}{\sqrt{1-\lambda^2 r^2}}.\nonumber
\end{align}

When the graph is a Pansu sphere $\mathcal{P}^n_\lambda$, by substituting \eqref{partialxy} into \eqref{dform}, we have
\begin{align}\label{pansudform}
D(p)=\frac{r}{\sqrt{1-\lambda^2r^2}},
\end{align}
and so by \eqref{pansudform} \eqref{pnormal1} \eqref{partialxy}, the p-normal vector at the point $p$ in $\mathcal{P}^n_\lambda$ can be written in terms of $r$, namely,
\begin{align}\label{pnormal2}
\tilde{N}(p)=\sum_{j=1}^{n}\left[(\lambda x_j +\frac{\sqrt{1-\lambda^2 r^2}}{r} y_j)\mathring{e}_{x_j}(p)+(\lambda y_j-\frac{\sqrt{1-\lambda^2 r^2}}{r}{x_j})\mathring{e}_{y_j}(p)\right].
\end{align}
By \eqref{parea}, \eqref{pansudform}, \eqref{Pansu}, and use the spherical coordinates on $\mathbb{R}^{2n}$, a straight computation shows that the p-area of the Pansu sphere in $\mathbb{H}_n$ is given by
\begin{align}\label{pansuarea}
\mathcal{A}(P^n_\lambda) &= \int_{p\in \mathcal{P}^n_\lambda} D(p)dx_1dy_1\cdots dx_ndy_n \\ \nonumber
&=2 \int_0^{\frac{1}{\lambda}}\int_{\mathbb{S}^{2n-1}} \frac{r^{2n}}{\sqrt{1-\lambda^2 r^{2}}}  dS^{2n-1}dr  \\ \nonumber
&= S_{2n-1}\frac{\sqrt{\pi}\Gamma(n+\frac{1}{2})}{\lambda^{2n+1}\Gamma(n+1)},
\end{align}
where $\Gamma(x)$ is the Gamma function and $S_{2n-1}$ is the (Euclidean) surface area of the $(2n-1)$-sphere in $\mathbb{R}^{2n}$. In particular, when $n=1$, $\mathcal{A}(\mathcal{P}^1_\lambda)=\frac{\pi^2}{\lambda^3}$.

\section{The proofs}\label{allproofs}
\begin{proof}[Proof of Proposition \ref{anydirection}]
We point out that although the proof below is only for $n=1$, the same argument can be applied to the higher dimensional Heisenberg groups.

Let $q=(r\cos\theta, r\sin\theta, h(r))$ be a point on $\Sigma$. Denote $h^+, h^-$ by the graph of $\{h\geq 0\}$ and $\{h<0\}$ respectively and $h_r^\pm:=\frac{dh^\pm}{dr}$. Use \eqref{parea}, \eqref{dform}, \eqref{pnormal1} for $n=1$ and set $\eta:=\cos\alpha-r\sin\theta\sin\alpha\cos\beta+r\cos\theta\sin\alpha\sin\beta$, we have
\begin{align*}
|L_{q_*}u(0)\cdot N(q)| d\Sigma_q&=\Big|\big(\sin\alpha\cos\beta \mathring{e}_1(q)+\sin\alpha\sin\beta \mathring{e}_2(q) +\eta T(q)\big)\\
&\hspace{1cm} \cdot \frac{-1}{D(q)}\big( ({h^+_r}\cos\theta -r\sin\theta)\mathring{e}_1(q)+({h^+_r}\sin\theta+r\cos\theta)\mathring{e}_2(q) \big)\Big|D(q)r dr d\theta\\
&=r|\sin\alpha||{h^+_r}\cos(\theta-\beta)-r\sin(\theta-\beta)| dr d\theta \\
&=r|\sin\alpha|\sqrt{({h^+_r})^2 +r^2}\Big|\big( \frac{{h^+_r}}{\sqrt{({h^+_r})^2+r^2}}\cos(\theta-\beta)-\frac{r}{\sqrt{({h^+_r})^2+r^2}}\sin(\theta-\beta)\big)\Big|dr d\theta \\
&=r |\sin\alpha|  \sqrt{({h^+_r})^2+r^2} \big| \cos(\theta-\beta+\phi)\big| dr d\theta,
\end{align*}
where $\cos\phi=\frac{{h^+_r}}{\sqrt{h_r^2-r^2}}$. We also have the similar result for the function $h^-$. Since the angles $\phi$ and $\beta$ both are independent of the angle $\theta$, by the continuity of $h^+_r$ and $h^-_r$ the projected p-area of $\Sigma$ is given by
\begin{align}\label{anydirection3}
\mathcal{A}(\Sigma|u(0)^\perp)&=\iint_{q\in \Sigma}  |L_{q_*}u(0)\cdot N(q)| d\Sigma_q \\
&=|\sin\alpha| \int_{0}^{R} r \sqrt{(h^+_r)^2+r^2} \int_{0}^{2\pi} |\cos(\theta-\beta+\phi)|  d\theta dr \nonumber \\
&\hspace{2cm}+|\sin\alpha| \int_{0}^{R} r \sqrt{(h^-_r)^2+r^2} \int_{0}^{2\pi} |\cos(\theta-\beta+\phi)|  d\theta dr  \nonumber  \\
&=4 |\sin\alpha| \Big( \int_0^R r\sqrt{(h^+_r)^2 +r^2} dr+\int_0^R r\sqrt{(h^-_r)^2 +r^2} dr\Big).  \nonumber
\end{align}
If the functions $h_r^+$ and $h_r^-$ are continuous on $[0,R]$, then the both integrals on the right-hand side of \eqref{anydirection3} are finite, and so $\mathcal{A}(\Sigma|u(0)^\perp)=C |\sin\alpha|$ for some constant $C$, depending only on $h^+(r)$ and $h^-(r)$. Otherwise, if any of $|h^+(r)|$ and $|h^-(r)|$ is bounded by $\frac{r}{\sqrt{R^2-r^2}}$ on the interval $[0,R)$, say $h^+(r)$, then we may choose a positive number $M>\sqrt{1+R^2}$ such that the following estimate holds:
\begin{align}\label{integralbounded}
\int_0^R r\sqrt{(h^+_r)^2+r^2} dr
&\leq \int_0^R r\sqrt{\frac{r^2}{R^2-r^2}+r^2} dr =\int^R_0 r^2 \sqrt{\frac{1+R^2-r^2}{R^2-r^2}} dr \\
&< R^2 M\int_0^R \sqrt{\frac{1}{R^2-r^2}}dr  = R^2 M \lim_{a\rightarrow R^-} \tan^{-1}\big(\frac{r}{\sqrt{R^2-r^2}} \big)\Big|_0^a =\frac{R^2M\pi}{2}<\infty.  \nonumber
\end{align}
Therefore, by the assumptions for $h^+(r), h^-(r)$, and \eqref{anydirection3}, we conclude that the projected p-area $\mathcal{A}(\Sigma|u(0)^\perp)< C |\sin\alpha|$
for some constant $C$ only depending on the function $h^\pm(r)$.

When $\alpha=\frac{\pi}{2}$, the vector $u(0)$ is a unit vector on the $xy$-plane, which is the pushforward of a p-normal vector $\tilde{N}(p)$ for some point $p$ on the Pansu sphere $\mathcal{P}^1_\lambda$. Thus, we have $u(0)=L_{{p^{-1}}_*}\tilde{N}(p)$ and the result follows immediately.
\end{proof}


In the following calculation, $p$ always stands for a fixed point on the Pansu sphere $\mathcal{P}^n_\lambda$ and the point $q$ is arbitrary on $\mathcal{P}^n_\lambda$. Notice that since the p-normal vector $\tilde{N}(p)$ of $\mathcal{P}^n_\lambda$ is on the contact plane $\xi_p$ at $p$, $\tilde{N}(p)$ can be carried parallel to the vector $L_{{qp^{-1}}_{*}}\tilde{N}(p)\in \xi_q$ by the pushforward $L_{{qp^{-1}}*}$ of the left translation $L_{qp^{-1}}:=L_q\circ L_{p^{-1}}$.

\begin{proof}[Proof of Proposition \ref{keyprop}]
First, we calculate the inner product of $L_{{qp^{-1}}_{*}}\tilde{N}(p)\cdot \tilde{N}(q)$ at $q\in \mathcal{P}_\lambda^n$ with respect to the Levi metric. Write $p=(x_j, y_j,f(x_j,y_j))$, $q=(\bar{x_j},\bar{y_j},f(\bar{x}_j, \bar{y}_j))$, where $f$ is the function defined by $\eqref{Pansu}$. Without loss of generality, we may assume that $q$ is on the upper half $\{f\geq 0\}$ of $\mathcal{P}^n_\lambda$. Let $r$ (resp. $\bar{r})$ denote the radial of the spherical coordinates of the points $(x_1,y_1, \cdots, x_n,y_n)$ (resp. $(\bar{x}_1, \bar{y}_1, \cdots, \bar{x}_n, \bar{y}_n)$) $\in \mathbb{R}^{2n}$. Since $\{ \mathring{e}_{x_j}(p), \mathring{e}_{y_j}(p)\}$ is an orthonormal basis in the contact plane $\xi_p$ for any $p\in\mathbb{H}_n$, by \eqref{pnormal2} and \eqref{pansudform}, we have
\begin{align}\label{npqnqdp}
&\indent|L_{{qp^{-1}}_{*}}\tilde{N}(p)\cdot \tilde{N}(q)|D(q)\nonumber \\
&=\Big|\sum_{j=1}^{n}\Big[(\lambda x_j +\frac{\sqrt{1-\lambda^2 r^2}}{r} y_j)\mathring{e}_{x_j}(q)+(\lambda y_j-\frac{\sqrt{1-\lambda^2 r^2}}{r}{x_j})\mathring{e}_{y_j}(q)\Big]\\ \nonumber
\hspace{3cm}&\cdot \sum_{j=1}^{n}\Big[(\lambda \bar{x}_j +\frac{\sqrt{1-\lambda^2 \bar{r}^2}}{\bar{r}} \bar{y}_j)\mathring{e}_{\bar{x}_j}(q)+(\lambda \bar{y}_j-\frac{\sqrt{1-\lambda^2 \bar{r}^2}}{\bar{r}}{\bar{x}_j})\mathring{e}_{\bar{y}_j}(q)\Big]\Big|\cdot \frac{\bar{r}}{\sqrt{1-\lambda^2\bar{r}^2}} \\ \nonumber
&=\Big|\sum_{j=1}^{n}\Big[ (\lambda x_j+\frac{\sqrt{1-\lambda^2 r^2}}{r}y_j)(\lambda \bar{x}_j+\frac{\sqrt{1-\lambda^2 \bar{r}^2}}{\bar{r}}\bar{y}_j)\\ \nonumber
\hspace{3cm}&+(\lambda y_j-\frac{\sqrt{1-\lambda^2 r^2}}{r}{x_j})(\lambda \bar{y}_j-\frac{\sqrt{1-\lambda^2 \bar{r}^2}}{\bar{r}}\bar{x}_j)\Big]\Big| \cdot \frac{\bar{r}}{\sqrt{1-\lambda^2\bar{r}^2}}\\ \nonumber
&=\Big|\sum_{j=1}^{n}\Big[ (\frac{\lambda^2 \bar{r}}{\sqrt{1-\lambda^2 \bar{r}^2}}+\frac{\sqrt{1-\lambda^2 r^2}}{r})(x_j\bar{x}_j+y_j\bar{y}_j) +(-\frac{\lambda \bar{r}}{r}\frac{\sqrt{1-\lambda^2 r^2}}{\sqrt{1-\lambda^2 \bar{r}^2}}+\lambda )(x_j\bar{y}_j-y_j\bar{x}_j)\Big]\Big|.
\end{align}
Observe that under the spherical coordinates, each of the components $x_j$ can be represented by
\begin{align*}
x_j=(\text{radial distance})\cdot (\text{product of sine and cosine functions})
\end{align*}
for $j=1,\cdots, n$. Thus, one can write the components $x_j=r a_j$ for some product $a_j$ of trigonometric functions. Similarly, one has $y_j=r b_j$, $\bar{x}_j=\bar{r} \bar{a}_j$, and $\bar{y}_j=\bar{r}\bar{b}_j$. Notice that the points $(a,b)=(a_1, b_1, \cdots, a_n,b_n)$ and $( \bar{a},\bar{b})=(\bar{a}_1,\bar{b}_1,\cdots, \bar{a}_n,\bar{b}_n)$ both are on the unit sphere $\mathbb{S}^{2n-1}\subset \mathbb{R}^{2n}$. Next, substitute $x_j$, $y_j$, $\bar{x_j}$, $\bar{y_j}$ by $r$, $\bar{r}$, $a_j$, $b_j$, $\bar{a}_j$, $\bar{b}_j$ in the last equation of \eqref{npqnqdp} and we get
\begin{align}\label{npnqdp}
|L_{{qp^{-1}}_{*}}\tilde{N}(p)\cdot \tilde{N}(q)|D(q)=\Big|\sum_{j=1}^{n}\Big[A(a_j\bar{a}_j+b_j\bar{b}_j)+B(a_j\bar{b}_j-b_j\bar{a}_j) \Big]\Big|,
\end{align}
where $A=\frac{\lambda^2 \bar{r}^2 {r}}{\sqrt{1-\lambda^2 \bar{r}^2}}+\bar{r}\sqrt{1-\lambda^2 r^2}$ and $B=-\frac{\lambda \bar{r}^2 \sqrt{1-\lambda^2 r^2}}{\sqrt{1-\lambda^2 \bar{r}^2}}+\lambda r \bar{r}$. We emphasis a key observation that
\begin{align}\label{key}
A^2+ B^2 = \frac{\bar{r}^2}{1-\lambda^2 \bar{r}^2},
\end{align}
which is independent of the choice of the point $p$.

Now consider the unit vectors $u=(a_1,b_1,\cdots, a_n, b_n)$ and $\bar{v}=(\bar{a}_1, \bar{b}_2, \cdots, \bar{a}_n,\bar{b}_n)$ in $\mathbb{R}^{2n}$ with respect to the usual Euclidean norm. Let $J:\mathbb{R}^{2n}\rightarrow \mathbb{R}^{2n}$ be the canonical almost complex structure defined by $J(u)=(-b_1, a_1,\cdots, -b_n, a_n)$. If $\theta$ is the angle between two unit vectors $u, \bar{v}$ (in the usual Euclidean norm) and denote $u \star \bar{v}$ by the Euclidean inner product of $u$ and $\bar{v}$, then we have $u \star \bar{v}=\cos\theta$ and $J(u) \star \bar{v}=\pm\sin\theta$. The sign depends on the chosen orientation of $\mathbb{R}^{2n}$. Therefore, the right-hand side of \eqref{npnqdp} can be represented as
\begin{align}\label{rotation}
\Big|\sum_{j=1}^{n}\Big[A(a_j\bar{a}_j+b_j\bar{b}_j)+B(a_j\bar{b}_j-b_j\bar{a}_j) \Big]\Big|
&=|A( u\star \bar{v})+B(J(u) \star\bar{v})|\\ \nonumber
&=\sqrt{A^2+B^2}\Bigg|\frac{A}{\sqrt{A^2+B^2}}\cos\theta_q \pm \frac{B}{\sqrt{A^2+B^2}}\sin\theta_q\Bigg|\\
&\stackrel{\eqref{key}}{=}\frac{\bar{r}}{\sqrt{1-\lambda^2 \bar{r}^2}}|\cos(\theta_q\mp \alpha)| \nonumber
\end{align}
for some value $\alpha=\alpha(r,\bar{r}, \lambda)$ such that $\cos(\alpha)=\frac{A}{\sqrt{A^2+B^2}}$. To our purpose we add the subscript $\theta_q$ to indicate that the angle $\theta$ is a function of the point $q$.

Recall that the p-area form $d\Sigma_q$ at the point $q$ in $\mathbb{H}_n$ for any graph $z=f(x,y)$ is
$d\Sigma_q=D(q) dx_1 dy_1\cdots dx_n dy_n$ (see \cite[equation (2.7)]{CHY}). Finally, combine \eqref{npnqdp}, \eqref{rotation} and use the integral for spherical coordinates, one has
\begin{align*}
\int_{q\in \mathcal{P}^n_\lambda} |L_{{qp^{-1}}_{*}}\tilde{N}(p)\cdot \tilde{N}(q)| d\Sigma_q
&=\int_{q\in \mathcal{P}^n_\lambda}  |L_{{qp^{-1}}_{*}}\tilde{N}(p)\cdot \tilde{N}(q)| D(q) dx_1 dy_1\cdots dx_n dy_n\\
&=\int_{p\in \mathcal{P}^n_\lambda}  \Big|\sum_{j=1}^{n}\Big[A(a_j\bar{a}_j+b_j\bar{b}_j)+B(a_j\bar{b}_j-b_j\bar{a}_j) \Big]\Big| dx_1 dy_1\cdots dx_n dy_n\\
&=2\int_{0}^{\frac{1}{\lambda}} \int_{0}^{2\pi} \frac{\bar{r}\cdot \bar{r}^{2n-1}}{\sqrt{1-\lambda^2 \bar{r}^2}}|\cos(\theta_q\mp \alpha)| dS_q d\bar{r}\\
&=2\int_{0}^{\frac{1}{\lambda}} \frac{\bar{r}^{2n}}{\sqrt{1-\lambda^2 \bar{r}^2}} \int_{0}^{2\pi} |\cos(\theta_q\mp \alpha)| dS_q d\bar{r} \\
&\stackrel{(*)}{=}4\omega_{2n-1} \int_{0}^{\frac{1}{\lambda}} \frac{\bar{r}^{2n}}{\sqrt{1-\lambda^2 \bar{r}^2}} d\bar{r}\\
&=2\frac{\sqrt{\pi}\omega_{2n-1}}{\lambda^{2n+1}}\frac{\Gamma(n+\frac{1}{2})}{\Gamma(n+1)},
\end{align*}
where $dS_q$ is the $(2n-1)$-dimensional (Euclidean) surface area element at $q$ and $\omega_{2n-1}$ is the volume of the unit $(2n-1)$-ball in $\mathbb{R}^{2n-1}$. Note that in $(*)$ above, for any fixed $\bar{r}\in (0,\frac{1}{\lambda}]$, (equivalently, $\alpha$ is a constant w.r.t. $q$), Lemma \ref{KR} implies that the integral $\int_0^{2\pi} |\cos(\theta_q \mp \alpha)|dS_q=2\omega_{2n-1}$, the projected area of the $(2n-1)$-dimensional sphere onto the $(2n-1)$-subspace in $\mathbb{R}^{2n}$. We left the last integral with respect to $\bar{r}$ to the reader and complete the proof.
\end{proof}

\begin{remark}\label{geointerprtation}
In the Euclidean space the projected area of $K$ onto the subspace $u^\perp$ can be represented as the integral of the inner product $u\star v$ as shown in \eqref{projecteda}. Suppose the unit vector $u$ is based at the origin and ends at somewhere on $\mathbb{S}^{n-1}$. The proof of Lemma \ref{KR} uses the fact that, by parallel transports in $\mathbb{R}^n$, the base point of $u$ is parallel moved to any point $p\in \mathbb{S}^{n-1}$ and make the inner product with the unit outward normal vectors $v(p)$ at $p$. Similarly, in $\mathbb{H}_n$ the integral of $|L_{qp^{-1}*}\tilde{N}(p) \cdot \tilde{N}(q)|$ over $\mathcal{P}^n_\lambda$ in Proposition \ref{keyprop} can be regarded as the projected p-area of $\mathcal{P}^n_\lambda$ onto the subspace $\tilde{N}(q)^\perp:= \{ \text{vector } u\in T_q\mathbb{H}_n| \  u\cdot N(q)=0 \text{ w.r.t. Levi metric}\}$. For instance, in $\mathbb{H}_1$, we have the orthonormal basis $\{e_1(q), N(q), T\}$ for $e_1(q)\in T_q\Sigma \cap \xi_q$, $N(q)=J(e_1(q))\in \xi_q$, and $T(q)=\frac{\partial}{\partial z}|_q$. Then $N(q)^\perp = span_q\{e_1(q), T(q)\}$. Geometrically $N(q)^\perp$ is a plane (Euclidean) perpendicular to the $xy$-plane.
\end{remark}

The next lemma shows that the map $L_{{p^{-1}}_*}$ is surjective.
\begin{lem}\label{surjective}
Let $\mathcal{P}^n_\lambda$ be the Pansu sphere in the Heisenberg group $\mathbb{H}_n$ defined by \eqref{Pansu} and denote $\xi_0$ by the contact plane at the origin in $\mathbb{H}_n$. For any $p\in \mathcal{P}_\lambda^n$ the map
\begin{align*}
p\mapsto N(0):=L_{{p^{-1}}_*}(\tilde{N}(p))\subset \mathbb{S}^{2n-1}\subset \xi_0
\end{align*}
which assigns to any unit p-normal vector $\tilde{N}(p)$ at $p\in \mathcal{P}^n_\lambda$ a unit p-normal vector $N(0)$ in $\mathbb{S}^{2n-1}$, is surjective.
\end{lem}
\begin{proof}
For any unit vector $N(0)$ on $\mathbb{S}^{2n-1}\subset \xi_0$ in $\mathbb{H}_n$, one may assume that $N(0)=(x_1,y_1,\cdots, x_n,y_n,0)$ with $\sum_{j=1}^n (x_j^2+y_j^2)=1$. Set $p=\frac{1}{\lambda}N(0)$. Then the point $p$ is in the intersection $\mathcal{P}^n_\lambda \cap \{ z=0\}$. Construct the unit p-normal $\tilde{N}(p)$ by replacing $x_j's$ and $y_j's$  on right-hand side of \eqref{pnormal2} by $\frac{1}{\lambda}x_j's$ and $\frac{1}{\lambda}y_j's$, respectively. It is clear that the vector $\tilde{N}(p)$ is the p-normal vector at $p\in \mathcal{P}^n_\lambda$ satisfying $L_{{p^{-1}}_*}\tilde{N(p)}=N(0)$.
\end{proof}

We mention that the map in Lemma \ref{surjective} is a generalization of the Gauss map defined by Chiu-Ho \cite{Chiuho} in $\mathbb{H}_1$. In that paper the authors studied the degree of the Gauss map for horizontally regular curves with some topological properties (see Theorem 1.6 and Theorem 2.1 in \cite{Chiuho}).

\begin{proof}[Proof of Theorem \ref{mainthm1}]
For any p-normal vector $N(q)$ at $q\in \Sigma$ one has that $L_{{q^{-1}}_*}N(q)\subset \mathbb{S}^{2n-1}$. Moreover, Lemma \ref{surjective} implies that there exists a p-normal vector $\tilde{N}(q')$ at some $q'\in \mathcal{P}^n_\lambda$ such that $L_{{q'q^{-1}}_*}N(q)=\tilde{N}(q')$, equivalently,
$$L_{{q^{-1}}_*}N(q)=L_{{q'^{-1}}_*}\tilde{N}(q').$$
Thus, for any fixed p-normal vector $\tilde{N}(p)$ at $p\in \mathcal{P}^n_\lambda$ we have
\begin{align}\label{surjective2}
|L_{qp^{-1}*}\tilde{N}(p) \cdot N(q)|&=| L_{{p^{-1}}_*}\tilde{N}(p)\cdot  L_{{q^{-1}}_*}N(q)|\\
&=| L_{{p^{-1}}_*}\tilde{N}(p)\cdot  L_{{q'^{-1}}_*}\tilde{N}(q')| \nonumber \\
&=| L_{{q'p^{-1}}_*}\tilde{N}(p)\cdot  \tilde{N}(q')|.\nonumber
\end{align}
Finally, use \eqref{surjective2}, Proposition \ref{keyprop}, one gets
\begin{align*}
\int_{p\in \mathcal{P}^n_\lambda} \mathcal{A}(\Sigma | \tilde{N}(p)^\perp) d\Sigma_p
&=\int_{p\in\mathcal{P}^n_\lambda} \int_{q\in\Sigma} |L_{qp^{-1}*}\tilde{N}(p) \cdot N(q)| d\Sigma_q d\Sigma_p \\
&= \int_{q\in\Sigma} \int_{p\in\mathcal{P}^n_\lambda}   |L_{qp^{-1}*}\tilde{N}(p) \cdot N(q)|  d\Sigma_p d\Sigma_q \\
&= \int_{q\in\Sigma} \int_{p\in\mathcal{P}^n_\lambda}   | L_{{q'p^{-1}}_*}\tilde{N}(p)\cdot  \tilde{N}(q')| d\Sigma_p d\Sigma_q \\
&=2 C_n \omega_{2n-1}  \int_{q\in\Sigma} d\Sigma_q \\
&=2 C_n \omega_{2n-1} \mathcal{A}(\Sigma)
\end{align*}
and complete the proof of \eqref{formula1}. The second result \eqref{formula2} can be immediately obtained by \eqref{pansuarea} and \eqref{formula1}.
\end{proof}

In the next proof, we only show the case in $\mathbb{H}_1$ for simplicity and the proofs for higher dimensions are same. When $n\geq 2$, the constant $C$ in Theorem \ref{rotationalpop} depends only on the dimension $n$ of $\mathbb{H}_n$.

\begin{proof}[Proof of Theorem \ref{rotationalpop} for $n=1$]
Let $\Sigma$ be any rotationally symmetric compact surface in $\mathbb{H}_1$ defined by $(r,\theta)\rightarrow (r\cos\theta, r\sin\theta,h(r))$, $0\leq r\leq R$ for some real number $R$ and $0\leq \theta \leq 2\pi$. Let $\mathcal{P}^1_\lambda$ be the Pansu sphere in $\mathbb{H}_1$ defined by \eqref{Pansu} and $p$ a fixed point on $\mathcal{P}^1_\lambda$. Denote $h^+, h^-$ by the graph of $\Sigma^+:=\{h\geq 0\}$ and $\Sigma^-:=\{h<0\}$ respectively. By using \eqref{parea}, \eqref{dform}, \eqref{pnormal1}, \eqref{pnormal2} for $n=1$ and write $p=(\tilde{r}\cos\theta, \tilde{r}\sin\theta, h^+(r))$, one has
\begin{align*}
|L_{{qp^{-1}}_*} \tilde{N}(p) \cdot N(q)| d\Sigma_q
&= \Big| L_{{qp^{-1}}_*} \big( (\lambda \tilde{r}\cos\theta+\sqrt{1-\lambda^2 \tilde{r}^2} \sin\theta )\mathring{e}_1(p)+(\lambda \tilde{r}\sin\theta -\sqrt{1-\lambda^2 \tilde{r}^2} \cos\theta)\mathring{e}_2(p)\big)\\
&\hspace{1cm}\cdot \frac{-1}{D(q)}\Big( (h^+_r\cos\phi -r\sin\phi)\mathring{e}_1(q)+(h^+_r\sin\phi+r\cos\phi)\mathring{e}_2(q)\Big)\Big|D(q)r dr d\phi \\
&=\Big| \big(  (\lambda \tilde{r}\cos\theta+\sqrt{1-\lambda^2 \tilde{r}^2}\sin\theta )\mathring{e}_1(q)+(\lambda \tilde{r}\sin\theta -\sqrt{1-\lambda^2 \tilde{r}^2} \cos\theta)\mathring{e}_2(q)\big) \\
&\hspace{1cm} \cdot \big( (h^+_r\cos\phi -r\sin\phi)\mathring{e}_1(q)+(h^+_r\sin\phi+r\cos\phi)\mathring{e}_2(q)\big)\Big| r dr d\phi \\
&=|(\lambda  h^+_r \tilde{r}-r\sqrt{1-\lambda^2\tilde{r}^2} )\cos(\theta-\phi) + (\lambda r \tilde{r} +h^+_r \sqrt{1-\lambda^2 \tilde{r}^2})\sin(\theta-\phi)|r dr d\phi
\end{align*}
Set $A:=\lambda  h^+_r \tilde{r}-r\sqrt{1-\lambda^2\tilde{r}^2}$ and $B:=\lambda r \tilde{r} +h^+_r \sqrt{1-\lambda^2 \tilde{r}^2}$, and so $A^2+B^2=(h^+_r)^2 +r^2$. We point out that the term $A^2+B^2$ does not involve any terms about $\tilde{r}$ and $\theta$, namely, it is independent of the choices of the p-normal $\tilde{N}(p)$ of the Pansu sphere. Now we set $\cos\psi =\frac{A}{\sqrt{A^2+B^2}}$ and $\sin\psi=\frac{B}{\sqrt{A^2+B^2}}$ and notice that the angle $\psi$ does not depend on the angle $\phi$ and so when taking the integral with respect to the angle $\phi$, the angle $\phi$ can be regarded as a constant. Then one has the following formula for the projected p-area
\begin{align}\label{rotasymm}
\mathcal{A}(\Sigma^+|\tilde{N}(p)^\perp)
&=\iint_{q\in \Sigma}|L_{{qp^{-1}}_*}\tilde{N}(p) \cdot N(q)| d\Sigma_q \\
&=\iint_{q\in \Sigma} r \sqrt{A^2+B^2} \Big| \frac{A}{\sqrt{A^2+B^2}} \cos(\theta-\phi)+\frac{B}{\sqrt{A^2+B^2}} \sin(\theta-\phi)\Big| dr d\phi  \nonumber  \\
&:=\iint_{q\in \Sigma} r \sqrt{(h^+_r)^2 +r^2}\big| \cos\psi \cos(\theta-\phi)+\sin\psi \sin(\theta-\phi)\big|  dr d\phi \nonumber  \\
&=\int_{0}^R r\sqrt{(h^+_r)^2+r^2} \Big( \int_{0}^{2\pi} |\cos(\theta-\phi-\psi)|  d\phi \Big) dr  \nonumber  \\
&=C\int_0^R r \sqrt{(h^+_r)^2+r^2} dr\nonumber
\end{align}
for some constant $C$. Here we have used the fact that the angles $\psi$ and $\theta$ are independent of $\phi$. By using the assumptions for the functions $h^\pm(r)$ and applying the same argument in Proposition \ref{anydirection} (below the equation \eqref{anydirection3}), we conclude that the projected p-area $\mathcal{A}(\Sigma|\tilde{N}(p)^\perp)=\mathcal{A}(\Sigma^+|\tilde{N}(p)^\perp)+\mathcal{A}(\Sigma^-|\tilde{N}(p)^\perp)$ is a constant as desired and complete the proof of Theorem \ref{rotationalpop}.
\end{proof}


Finally, as we have known that the Cauchy's surface area formula in the Euclidean spaces is a special case for a generalized notion, called querrmassintegrals, in Integral Geometry \cite{Santalo}. Recently, the author showed some results of Integral Geometry in $\mathbb{H}_n$ (see \cite{Huang} for Crofton's formula and the containment problems), it is still not clear that whether we can have the counterparts of querrmassintegrals in the Heisenberg groups $\mathbb{H}_n$ for any $n\geq 1$. It is also not clear that whether several concepts (for instance, the support functions in $\mathbb{R}^n$ for the convex bodies) related to that of convexity in Convex Geometry is capable of being developable in $\mathbb{H}_n$ or not. Those might be the interesting topics worth to develop for the future study.

\end{document}